\documentclass[12pt]{amsart}
\usepackage{amsmath}
\usepackage[margin=1.6cm]{geometry}
\usepackage{parskip}
\usepackage{enumerate}

\usepackage{graphicx}
\usepackage{import}
\usepackage{pstricks}
\usepackage{amsfonts}
\usepackage{array}
\usepackage{amsthm}
\usepackage{thmtools}
\usepackage{amssymb}

\usepackage{hyperref}
\usepackage{cleveref}

\newtheorem{theorem}{Theorem}[section]

\newtheorem{lemma}[theorem]{Lemma}
\newtheorem{proposition}[theorem]{Proposition}

\newcommand{\Irr}{\operatorname{Irr}}
\newcommand{\GG}{\mathbb{G}}

 \title[$\ell$-Modular Representation of  $\mathbb{G}(\mathcal{O}_2)$]{\textbf{THE $\ell$-MODULAR REPRESENTATION OF REDUCTIVE GROUPS OVER FINITE LOCAL RINGS OF LENGTH TWO}}
\author{Nariel Monteiro}
\address{Department of Mathematics, Tufts University, 503 Boston Ave, Medford, MA 02155}
\email{nariel.monteiro@tufts.edu}

\begin{document}

\maketitle

\begin{abstract}
Let $\mathcal{O}_2$ and $\mathcal{O}'_2$ be two distinct finite local rings of length two with residue field of characteristic $p$. Let $\mathbb{G}(\mathcal{O}_2)$ and $\mathbb{G}(\mathcal{O}'_2)$, be the group of points of any reductive group scheme $\mathbb{G}$ over $\mathbb{Z}$ such that $p$ is very good for $\mathbb{G} \times \mathbb{F}_q$. We prove that there exists an isomorphism of group algebra $K[\mathbb{G}(\mathcal{O}_2)] \cong  K[\mathbb{G}(\mathcal{O}'_2)]$, where $K$ is a sufficiently large field of  characteristic different from $p$.
\end{abstract}

 \section{Introduction}

Let $\mathcal{O}$ be a discrete valuation ring with a unique maximal ideal $\mathfrak{p}$ and having finite residue field $\mathbb{F}_{q}$, the field with $q$ elements where $q$ is a power of a prime $p$. We denote  $\mathcal{O}_r$ to be the reduction of $\mathcal{O}$ modulo $\mathfrak{p}^r$. Similarly, given $\mathcal{O}'$ a second discrete valuation ring  with the same residue field  $\mathbb{F}_{q}$, we can define $\mathcal{O}'_r$. The complex representation theory of general linear groups over the rings $\mathcal{O}_r$ has been heavily study \cite{stasinski2017representations}. In particular, it has been conjectured by Onn \cite{OnnUri2008Roag} that there is an isomorphism of group algebras $\mathbb{C}[\mathbb{GL}_n(\mathcal{O}_r)]\cong\mathbb{C}[\mathbb{GL}_n(\mathcal{O}'_r)]$. When $r=2$, this conjecture was proven by Singla \cite{SinglaPooja2010Orog}. Moreover, assuming $p$ is odd, Singla proved a generalization of the conjecture for $r=2$ when $\GG$ is either $\mathrm{SL}_n$ with $p \nmid n$ or the following classical groups $\mathrm{Sp}_n$, $\mathrm{O}_n$, and $\mathrm{U}_n$ \cite{SinglaPooja2012ORoC}. Later on, Stasinski proved it for  $\mathrm{SL}_n$ for all $p$ \cite{StasinskiAlexander2019Romo}. More generally, for a possibly non-classical reductive group, Stasinski and Vera-Gajardo compared the complex representation of the two groups in question \cite{StasinskiAlexander2019Rorg}. 

Let us briefly describe the result of Stasinski and Vera-Gajardo. 
First recall that for any group scheme $\GG$ over $\mathbb{Z}$ and any commutative ring $R$, we may speak of the group $\mathbb{G}(R)$ of $R$-points.
We now let $\GG$ be a reductive group scheme over $\mathbb{Z}$ -- for details on such group schemes see e.g. \cite{demazure11:MR2867622} and \cite{jantzen03:MR2015057}.
Note that for any root datum -- see e.g. \cite{jantzen03:MR2015057} -- there is a split
reductive group scheme $\GG$ over $\mathbb{Z}$ with this root datum \cite[Exp. XXV Theorem 1.1]{demazure11:MR2867622}. The group $\GG \times \mathbb{F}_q$ obtained
by base-change with the finite field $\mathbb{F}_q$ of characteristic $p$ is a reductive algebraic group over $\mathbb{F}_q$, and we want to insist that
$$\mathbf{(*)} \quad \text{$p$ is \emph{very good} for $\GG \times \mathbb{F}_q$
;}$$ we observe that the condition $\mathbf{(*)}$ depends only on the \emph{root datum} of $\GG$ -- see \cite[Section 4]{springer70:MR0268192} for the definition of good/very good primes.

Under assumption $\mathbf{(*)}$, Staskinski and Vera-Gajardo proved that  $\mathbb{C}[\mathbb{G}(\mathcal{O}_2)]\cong\mathbb{C}[\mathbb{G}(\mathcal{O}'_2)]$. When $\GG \times \mathbb{F}_q$ is a classical absolutely simple algebraic group not of type $A$, any $p>2$ is very good for $\GG.$ When $G = \operatorname{SL}_n$ or $\operatorname{SU}_n$, $p$ is very good for $\GG$ if and only if $n \not \equiv 0 \pmod p.$ If $\GG \times \mathbb{F}_q$ is absolutely simple of exceptional type, any $p>5$ is \emph{very good} for $G$; see  \cite[4.3]{springer70:MR0268192} for the precise conditions when $p \le 5$.

 We study the $\ell$-modular representation of such group scheme over $\mathcal{O}_2$ a local ring of length two with finite residue field. One can show that $\mathcal{O}_2$ is isomorphic to one of the following rings: $\mathbb{F}_{q}[t]/t^{2}$ or $W_{2}(\mathbb{F}_{q})$, the ring of Witt vectors of length two over $\mathbb{F}_{q}$  \cite [Lemma 2.1]{StasinskiAlexander2019Rorg}. Thus, we let $G_2=\mathbb{G}(\mathcal{O}_2)$ and $G'_2=\mathbb{G}(\mathcal{O}'_2)$, where $\mathbb{G}$ stands for a reductive group scheme over $\mathbb{Z}$ such that $p$ is very good for $\mathbb{G} \times \mathbb{F}_q$.

In this paper, we generalize the previous results over $\mathbb{C}$ to results over a sufficiently large field $K$ of characteristic $l \neq p$. More precisely, we prove that there exists an isomorphism of group algebras $KG_2 \cong KG'_2$ over a sufficiently large field $K$ of characteristic $l$ as long as $l\neq p$. We take $K$ to be a sufficiently large so that the representation theory of the groups over $K$ is the same as the representation theory over an algebraically closed field of characteristic $l$. 

In order for us to understand those two group algebras, we study their decomposition by block algebras. Given $A$ a $K$-algebra, we define a block of $A$ to be a primitive idempotent $b$ in the center of $A$, which is denoted as $Z(A)$; the algebra $Ab$ is called a block algebra of $A$. By an idempotent, we mean an element $b$ such that $b^2=b$ and primitive if $b = b_1 + b_2$ is an expression of $b$ as a sum of idempotents such that
$b_1b_2 = 0$, either $b_1 = 0$ or $b_2 = 0$. Moreover, the block algebra $Ab$ is an indecomposable two-sided ideal summand of $A$. Thus for a finite group $G$, we can write $$KG=B_1 \oplus B_2 \oplus \cdots \oplus B_n$$
where the $B_i=e_iKG$ are the unique block algebras of $KG$ up to ordering \cite{craven2019representation}. Moreover $e_i\cdot e_j=0$ whenever $i \neq j$.

We investigate blocks of $KG$ which we denote as $Bl(G)$. To understand the block algebra of those two group algebras, we exploit the fact that those two groups are extensions of $\mathbb{G}(\mathbb{F}_{q})$ by an abelian $p$-group denoted as  $N$ which is isomorphic to the Lie algebra of $\mathbb{G}(\mathbb{F}_{q})$ denoted as $\mathfrak{g}$ \cite [Lemma 2.3.]{StasinskiAlexander2019Rorg}. In fact, for the rest of this section let $G$ denote either $G_2$ or $G'_2$. Let the map $$\rho:G \to  \mathbb{G}(\mathbb{F}_{q})$$ be the surjective map obtained from the map $\mathcal{O}_2 \to \mathbb{F}_{q} $ with $N=ker(\rho)$. The two groups $G_2$ and $G'_2$  act on $\mathfrak{g}$, via its quotient $\mathbb{G}(\mathbb{F}_{q})$. This action is transformed by the above isomoprhism into the action of $G$ on its normal subgroup $N$, we explore the details of this action in section 3. In section 2, we use Clifford's theory to relate blocks of $G$ with blocks of $N$. Specifically, we have that any block $b \in Bl(G)$ is equal to $Tr^{G}_{H}(d)=\sum\limits_{g \in {G/H}}gdg^{-1}$ for $d$ in $Bl(H)$, where $H$ is the stabilizer  in $G$ for some block of $N$. In section 4, we exploit the fact that Clifford theory, in combination with the results obtained in the characteristic zero case, gives an isomorphism between certain blocks algebra of the stabilizer of those two groups. In section 5, we induce the previous isomorphism to an isomorphism between block algebras of $G_2$ and $G'_2$ by taking advantage of the fact that blocks of a group algebra are interior $G$-algebras. This isomorphism will then give rise to an isomorphism between those two group algebras. 

\subsection*{Acknowledgements}

I am extremely grateful to my advisor George McNinch for suggesting this project and his invaluable guidance and support throughout it. I thank Alexander Stasinski for reading a first draft of this article and providing very valuable and constructive feedback. This research was partially supported by Thesis Writing Fellowship, Noah Snyder's NSF CAREER Grant DMS-1454767.

 \section{Background on Clifford theory of blocks}

 Let $G$ be a finite group and $N$ a normal subgroup of $G$. We use Clifford theory to relate  blocks of $KG$ with blocks of $KN$. Given $b$ a block of $KG$ and $e$ a block of $KN$, we say $b$ covers $e$ if $be \neq 0$. We denote the set of blocks of $G$ that covers $e$ by $Bl(G|e)$. Note that there is a natural action of $G$ on the set of blocks of $N$ by conjugation. Moreover one can show that for a fixed $b$, the set of blocks $e$ of $KN$ satisfying $be \neq 0$ is a $G$-conjugacy class of blocks of $KN$ \cite[Proposition 6.8.2]{linckelmann2018block}.
 
The following Clifford theorem for blocks of $G$ holds:

\vspace{5mm}
 
\begin{theorem}
\label{theorem:Clifford}
Let $G$ be a finite group and $N$ a normal subgroup of $G$. Given a block $b \in Bl(G)$, we have:

\begin{enumerate}
    \item there exists a block $e \in Bl(N)$ such that $b\cdot e \neq 0$ so $b \in Bl(G|e)$
    
    \item if $e$ is as in (1), then $b=Tr_H^G(d):=\sum\limits_{g \in {G/H}}gdg^{-1} $ for a unique block $d \in Bl(H|e)$ where $H$ is the stabilizer of $e$ in $G$ under the conjugation action.
    
    \item The assignment $d\mapsto Tr_H^G(d)=b$ gives a bijection between $Bl(H|e)$ and $Bl(G|e)$.
    
\end{enumerate}
 
\end{theorem}

For more information about Clifford theory of blocks consult Linckelmann \cite{linckelmann2018block}, Nagao \cite{nagao2014representations}, or Craven \cite{craven2019representation}. Thus in order for us to understand the block algebra of those two groups, we start by investigating the block structure of $eKH$ for a fixed block $e \in Bl(N)$, where as above $H$ is the stabilizer of $e$ in $G$, under the conjugation action.
 
 Let $\Irr_K(N)$ be the set of irreducible characters of $KN$. Since $l$ does not divide $|N|$ (because N is a p-group), any block of $KN$ is defined by  $e=e_\chi= \frac{\chi(1)}{|N|}\sum\limits_{g \in N}\chi(g^{-1}) g$, for some $\chi$ in $\Irr_K(N)$ \cite{navarro1998characters}. Specifically, $KN$ is a semisimple algebra and $eKN$ is a matrix algebra, so it has defect zero. Thus, we can take advantage of the following proposition to show that $eKH \cong eKN \otimes_K K^{\alpha^{-1}}L$. For an $\alpha$ in $Z^2(G;K^*)$, we denote $K^{\alpha}G$ as the twisted group algebra of $G$ by $\alpha$. It has basis as $K$-algebra the elements of $G$ and  given $x,y \in G$, we define $x \cdot y = \alpha(x, y)xy$
 where $xy$ is the product of $x$ and $y$ in $G$. Furthermore, given a different  $\beta$ in $Z^2(G;K^*)$ there is a G-graded algebra isomorphism $K^{\alpha}G \cong K^{\beta}G$ if and only if the classes of $\alpha$ and $\beta$ are equal in $H^2(G;K^*)$ \cite{linckelmann2018block1}. Thus, for the purpose of this paper $\alpha$ is defined up to an element in the second cohomology group $H^2(G;K^*)$ with $G$ acting trivially on $K^*$.
 
\vspace{5mm}
 
  \begin{proposition}
  \label{proposition:Twisted}
   (Theorem 6.8.13, Linckelmann \cite{linckelmann2018block}) Let H be a finite group, N a normal subgroup of H and $e$ a $H$-stable block of defect zero of $KN$. Set $S = KNe$ and suppose that $K$ is a splitting field for S. Set $L = H/N$. For any $x \in H$ there is $s_x \in S^*$ such that $xtx^{-1} = s_xt (s_x )^{-1}$ for all $t \in S$ and such that $s_xs_y = s_{xy}$ if at least one of $x$, $y$ is in $N$. Then the 2-cocycle $\alpha \in Z^2(H;K^*)$ defined by $s_xs_y = \alpha(x, y)s_{xy}$ for $x, y \in N$ depends only on the images of $x, y$ in $L$ and induces a 2-cocycle, still denoted $\alpha$, in $Z^2(L;K^*)$, and we have an isomorphism of $K$-algebras 
  $$KHe \xrightarrow{\sim} S \otimes_K K^{\alpha^{-1}} L$$
sending $xe$ to $s_x \otimes \bar{x}$, where $x \in H $and $\bar{x} $ is the image of $x$ in $L$.
\end{proposition}

\section{The action of $G_2$ and $G'_2$ on the kernel $N$}

In this section, we study the action of $G$ on the kernel $N$ so we can understand the stabilizer $H$ mod $N$ and thus we can take advantage of the \Cref{proposition:Twisted}. The two groups $G_2$ and $G'_2$ act on $\mathfrak{g}$, via its quotient $\mathbb{G}(\mathbb{F}_{q})$. Particularly, there is a natural adjoint action of $\mathbb{G}(\mathbb{F}_{q})$ on $\mathfrak{g}$: $$\operatorname{Ad}:\mathbb{G}(\mathbb{F}_{q}) \to \operatorname{GL}(\mathfrak{g}).$$ 

Also there is an automorphism of $\sigma :\mathbb{G}(\mathbb{F}_{q}) \to \mathbb{G}(\mathbb{F}_{q})$ given by raising each matrix entry to the $p$-th power. Note that $\sigma$ composed with $\operatorname{Ad}$ gives rise to a second action of $\mathbb{G}(\mathbb{F}_{q})$ on $\mathfrak{g}$. Thus, given $X \in \mathfrak{g}$, we note that the action of $G_2=\mathbb{G}(\mathbb{F}_{q}[t]/t^{2})$ and $G'_2=\mathbb{G}(W_{2}(\mathbb{F}_{q}))$ are as follow:
$$g\cdot_{1}X=\operatorname{Ad}(\Bar{g})X  \quad \text{for}\, g \in G_2$$
$$g\cdot_{2}X=\operatorname{Ad}(\sigma(\Bar{g}))X \quad \text{for}\, g \in G'_2,$$ 

where $\Bar{g}=\rho(g)$, see \cite[Section 2.3]{StasinskiAlexander2019Rorg} for more details. 
Moreover, we also have actions of both $G_2$ and $G'_2$ on $\mathfrak{g}^*= \operatorname{Hom}_{\mathbb{F}_{q}}(\mathfrak{g},\mathbb{F}_{q})$ and $\Irr_K(\mathfrak{g})$. Given $F$ an element of $\mathfrak{g}^*$ or $\Irr_K(\mathfrak{g})$ and $X \in \mathfrak{g}$, we define: 
 $$(g\cdot_{1}F)(X)=F(g^{-1} \cdot_{1}X)  \quad \text{for}\, g \in G_2$$
 $$(g\cdot_{2}F)(X)=F(g^{-1} \cdot_{2}X)  \quad \text{for}\, g \in G'_2.$$
 \begin{lemma}
 
\label{lemma:Characters}
Let $\chi$ and $\tau$ be characters of $\mathfrak{g}$ then $^{g}\chi=\tau$ for some $g \in G'_2$ if and only if there is $h \in G_2$ such that $^{h}\chi=\tau$.
\end{lemma}

\begin{proof}
Let $g \in G'_2$ and $h \in \rho^{-1}(\sigma(\Bar{g}^{-1}) \subset G_2$ then we have:  $$^{g}\chi(X)=(g\cdot_{2}\chi)(X)=\chi(g^{-1} \cdot_{2}X)=\chi(\rho^{-1}(\sigma(\Bar{g}^{-1})) \cdot_{1}X)=\ ^{h}\chi(X).$$ 

\end{proof}

  Given a non-trivial irreducible character $\phi:\mathbb{F}_{q} \to K^*$ and for each $\beta \in \mathfrak{g}^*$, we define the character $\chi_{\beta} \in \Irr(\mathfrak{g})$ by 
 
 $$\chi_{\beta}(X)=\phi(\beta(X)).$$

\begin{lemma}
\label{lemma:character2}
\cite [Lemma 4.1]{StasinskiAlexander2019Rorg} The map $\beta \mapsto \chi_\beta$ defines an isomorphism of abelian group $\mathfrak{g}^* \to \Irr(\mathfrak{g})$ that is $G$-invariant i.e.  for $g \in G$ we have

$$g \cdot_{1} \beta \mapsto g \cdot_{1} \chi_\beta \quad \text{or} \quad  g \cdot_{2} \beta \mapsto g \cdot_{2} \chi_\beta. $$
\end{lemma}

Note, there is also an isomorphism of $\mathfrak{g} \cong N$ which is $G$-invariant \cite [Lemma 2.3.]{StasinskiAlexander2019Rorg}. Thus, we can identify the characters of $N$ with characters of $\mathfrak{g}$. Given a character $\chi$ of $N$, we define $H_2$ and $H'_2$ to be the stabilizer of $\chi$ in $G_2$  and $G'_2$ respectively. Thus, we have the following lemma about the stabilizer of $e_\chi$.

\begin{lemma}
\label{lemma:stabilizer}
The stabilizer $H_2/N$ and $H'_2/N$ of $e_\chi$ are isomorphic.
\end{lemma}

\begin{proof}
Since $g \cdot e_\chi=e_{g \cdot \chi}$, the stabilizer of $e_\chi$ is the same as the stabilizer of $\chi$ \cite[Section 9]{navarro1998characters}. Thus, we compute the stabilizer of  $\chi$. By the discussion above, we can consider $\chi$ to be in $\Irr(\mathfrak{g})$, by \Cref{lemma:character2} we have that $\chi=\chi_{\beta}$  for some $\beta \in \mathfrak{g}^* $. Thus: 

$$H'_2/N=\{g \in \mathbb{G}(\mathbb{F}_{q})|\rho^{-1}(g) \cdot_{2}\beta=\beta \}=\{g \in \mathbb{G}(\mathbb{F}_{q})|\rho^{-1}(\sigma(g)) \cdot_{1}\beta=\beta \}=\sigma^{-1}(H_2/N).$$

Since the map $\sigma$ is automorphism, it follows that $H_2/N \cong H'_2/N.$

\end{proof}

\section{The isomorphism of $eKH_2$ and $eKH'_2$}

In this section,  we prove that there is an isomorphism of blocks of $eKH_2$ and $eKH'_2$, where $e=e_\chi$ for $\chi$ in $\Irr_K(N)$ such that $H_2$ and $H'_2$ are the stabilizer of $e$ in $G_2$ and $G'_2$ respectively. To prove this isomorphism, we take advantage of \Cref{proposition:Twisted}. In order to do so, we need to understand the cofactors $\alpha$ associated to $eKH_2$ and $eKH'_2$. We prove that the cofactor $\alpha$ is trivial in both cases. For this, we introduce projective ($K$)-representation of $G$ with factor set $\alpha$. By which, we mean a map
$X:G\to \mathbb{GL}_n(K)$ such that $X(x)X(y) = \alpha(xy)X(xy)$
for all $x,y$ in $G$ , where $\alpha(xy)$ is in $K^*$. In fact, one can check that $\alpha$ is in $Z^2(G;K^*)$, where we assume that $G$ acts on $K^*$ trivially. Similarly, one can define a projective representation as a $K^{\alpha}G$-module \cite{nagao2014representations}. We let $H$ denote either $H_2$ or $H'_2$. The projective representations of $H$ are closely related to Clifford theory. In fact, we can use \cref{proposition:Twisted} to show that there is a  projective representation $V$ of $H$ that extends $\chi$. By that we mean that $V$ viewed as $KN$-module is isomorphic to the $KN$-module associated to $\chi$. 

\vspace{5mm}

\begin{proposition}
\label{proposition:Projective}

 The $\alpha $, obtained from \cref{proposition:Twisted}, associated to the block $e_\chi$ of $KN$ gives rise to a projective representation of $H$ that extends $\chi$.
\end{proposition}

\begin{proof}
Let $Y$ be the representation associated to $\chi$, i.e $Y:N \to \mathrm{GL}_{n}(K)$. Let $R$ be a set of representative of $L=H/N$ in $H$ and $S=KNe$. For each $x \in R$, by \cref{proposition:Twisted}, there is a $s_x \in S^*$ such that $xtx^{-1} = s_xt (s_x )^{-1}$ for all $t \in S$. Note also since $e$ is a $H$-stable block of defect zero of $KN$, then $S = KNe\cong \mathrm{M}_{n}(K)$. With abuse of notation, we can think of $s_x$ as an element of $\mathrm{GL}_{n}(K)$. Now define $X$ to be map $X:H \to \mathrm{GL}_{n}(K)$ such that for each $h=xn \in H$, we have $X(h)=s_xY(n)$ with $x \in R$ and $n \in N$. One can check that $X$ as defined above is a projective representation of $H$ with factor set  $\alpha$, in $Z^2(L;K^*)$ that extends $Y$.   
\end{proof}

In order to understand $\alpha$ from \cref{proposition:Twisted} which is associated to a projective ($K$)-representation of $G$, we study the projective representation of $H$ with factor set $\alpha$ that extends $\chi$ over a field of characteristic zero. To relate them, we need an $l$-modular system. By this we mean a triple $(F;R; K)$ where $F$ is a field of characteristic zero equipped with a discrete valuation, $R$ is the valuation ring in $F$ with maximal ideal $(\pi)$, and $K = R/(\pi)$ is the residue field of $R$, which is required to have characteristic $l$. If both $F$ and $K$ are splitting fields for $G$ we say that the triple is a splitting $l$-modular system for $G$. Note, we need an $l$-modular system so we can relate representation over a field $F$ of characteristic zero to representation over a field $K$ of characteristic $l$. The following lemma shows that given a projective representation over $F$, we can obtain a projective representation over $R$. Notice this lemma is just a generalization of the already known fact over group algebras that can be extended to hold over twisted group algebras. In fact the prove is the same, see \cite[Lemma 2.2.2] {schneider2012modular}. 

\vspace{5mm}

\begin{lemma}
\label{lemma:CDE}
Let $G$ be a finite group and  $M$ be a $F^\alpha G$-module  with $\alpha \in Z^2(G;R^*)$  then $M$ contains a lattice $L$ that is a  $R^\alpha G-module$. 
\end{lemma}

\begin{proof}
Note by \cite[Lemma 2.2.1] {schneider2012modular}  to show that $L$ is a lattice of $M$ it is enough to show that $L$
is finitely generated as an $R$-module and $L$ generates $M$ as a $F$-vector space. Thus, pick a $F$ basis $e_1,..., e_n$ of M then let $L':=Re_1+...+Re_n$ a lattice of M. Let $L:=\sum\limits_{g \in G}gL'$ then $L$ is a  $R^\alpha G-module$. Note $L$ is finitely generated as an $R$-module by $\{ ge_i : 1 \leq i \leq n,g \in G \}$ and it also generates $M$ as vector space. Thus L is a $G$-invariant lattice. 
\end{proof}

Before we introduce the following theorem, recall that a block of defect zero may be defined as a matrix algebra. Moreover, by the following \Cref{proposition:Webb} such a block is a ring summand of $KG$ which has a projective simple module.

\vspace{5mm}

\begin{proposition}
\label{proposition:Webb}
(Theorem 9.6.1, Webb \cite{webb2016course}). Let $(F;R; K)$ be a splitting $p$-modular system in which $R$ is complete, and let $G$ be a group of order $p^dq$ where $q$ is prime to $p$. Let $T$ be an $FG$-module of dimension $n$, containing a full RG-sublattice $T_0$. The following are equivalent:
\begin{enumerate}
    \item $p^d|n$ and $T$ is a simple $FG$-module.
    \item The homomorphism $RG \to \mathrm{End}_{R}(T_0)$ that gives the action of $RG$ on $T_0$ identifies $\mathrm{End}_{R}(T_0) \cong \mathrm{M}_{n}(R)$ with a ring direct summand of $RG$.
    \item  $T$ is a simple $FG$-module and $T_0$ is a projective $RG$-module.
    \item The homomorphism $KG \to \mathrm{End}_{K}(T_0/ \pi T_0)$ identifies $\mathrm{End}_{K}(T_0/ \pi T_0)\cong \mathrm{M}_{n}(K)$ with a ring direct summand of $KG$.
    \item As a $KG$-module, $T_0/ \pi T_0$ is simple and projective. 
\end{enumerate}

\end{proposition}

 Most importantly, by \Cref{proposition:Webb} a block of defect zero can have only one simple module and there is a unique ordinary simple module that reduces to it. This is used in the proof of the following theorem.

\vspace{5mm}

\begin{theorem}
\label{theorem:TCDE}

Given a $H$-stable block $e_\chi$ of defect zero of $KN$ where $N \trianglelefteq H$. Let $V$ be the unique ordinary simple module associated to this block. According to \Cref{proposition:Projective}, let $\hat{V}$ be the projective representation of $H$ that extends $V$, with cofactor $\hat{\alpha}$  in $Z^2(H/N;R^*)$. Any H-invariant lattice of $\hat{V}$ call it $L$, gives rises to a $K^{{\alpha}} H$-module that extends the simple projective $KN$-module associated to $e_\chi$, where $\alpha=\hat{\alpha} \mod{(\pi)}$.

\end{theorem}

\begin{proof}
Let $\hat{V}$ be as above so $\hat{V}$ is a $F^{\hat{\alpha}} H$-module such that  $\hat{V}{\downarrow}^H_N \simeq V$ as a $FN$-module. Now by \Cref{lemma:CDE} take $L$ to be an H-invariant lattice of $\hat{V}$ then consider its reduction to a $K^{{\alpha}} H$-module, call it $\bar{L}:=K \otimes_R L$. Note that $L$ is also an N-invariant lattice of $V$, the unique simple ordinary module associated to the block $e_\chi$. Thus, by \Cref{proposition:Webb} the reduction of $L$ is a simple module of $KN$. Therefore, $\bar{L}{\downarrow}^H_N$ is isomorphic to the simple module associated to this block. One can conclude that $\bar{L}$ is a $K^{{\alpha}} H$-module that extends the simple projective module associated to $e_\chi$, where $\alpha=\hat{\alpha} \mod{(\pi)}$.
\end{proof}

\vspace{5mm}

 \begin{proposition}
 \label{proposition:alpha}
  The $\alpha$ obtained from \Cref{theorem:TCDE} is trivial.
 \end{proposition}
 
 
 \begin{proof}
 Note by work of Stasinski and Vera-Gajardo, it was proven that any $\chi$ element of  $\Irr_F(N)$ extends to it is inertia group $H$ \cite[Proposition 4.5]{StasinskiAlexander2019Rorg}. Thus by \Cref{proposition:Projective}, the cofactor associated to the extension of $\chi$ is trivial over $F$, i.e. $\hat{\alpha}=1$. By \Cref{theorem:TCDE}, there is a $K^{{\alpha}} H$-module $\bar{L}$ that extends the simple projective module associated to $e_\chi$, where $\alpha=\hat{\alpha} \mod{(\pi)}=1$. Thus $\alpha$ is trivial.
 \end{proof}
 
 \vspace{5mm}
 
 We recall the following result from \Cref{lemma:stabilizer} that the stabilizer $H_2$ and $H'_2$ of $\chi$ are isomorphic mod $N$ i.e. $H_2/N \simeq H'_2/N$.
 We will denote $L:=H_2/N \simeq H'_2/N$ for this quotient.

\vspace{5mm}

\begin{theorem}
\label{theorem:Stabilizer}
The following $K$-Algebras are isomorphic $eKH_2$ and $eKH'_2$. 
\end{theorem}

\begin{proof}
We can apply \Cref{proposition:Twisted}, since $e$ is $H$-stable block of defect zero of $KN$. Thus, we have an isomorphism of $K$-algebras $$\Phi: eKH_2\cong eKN \otimes_K KL \cong eKH'_2$$ since by \Cref{proposition:alpha}, $\alpha^{-1}$ is trivial.  
\end{proof}

 \section{The isomorphism of $KG_2$ and $KG'_2$}

The following results about interior $G$-algebras will be useful in order to prove the isomorphism of those two group algebras. Note first that a $G$-algebra over a field $K$ is an $K$-algebra $A$ together with an action of $G$ on $A$ by $K$-algebra automorphisms. An interior $G$-algebra is a $G$-algebra where the action of $G$ is given by inner automorphism. The example to keep in mind is that $KG$ and $bKG$ are interior $G$-algebra with $G$ acting by conjugation, where $b \in Bl(G)$. 

Given $H$ a subgroup of $G$ and $B$ an interior $H$-algebra, we define $\mathrm{Ind^G_H}(B)$ to be the $K$-module $KG \otimes_{KH} B \otimes_{KH} KG$ and one can put an interior $G$-algebra structure on $\mathrm{Ind^G_H}(B)$. For more details on the interior $G$-algebra structure on $\mathrm{Ind^G_H}(B)$ one can consult Th{\'e}venaz's book on $G$-algebras and modular representation theory \cite{thevenaz1995g}. In fact, the following lemma shows that as a $K$-algebra one can think of $\mathrm{Ind^G_H}(B)$ as a matrix algebra over $B$.

\vspace{5mm}

\begin{lemma}
\label{lemma:Thevenaz} (Lemma 16.1, Th{\'e}venaz  \cite{thevenaz1995g}) Let $H$ be a subgroup of $G$ of index n, and $B$ be an interior $H$-algebra. Then we have $\mathrm{Ind^G_H}(B) \cong \mathrm{M}_{n}(B)$ as $K$-algebras.
\end{lemma}

\vspace{5mm}

In the following proposition, we will see that given certain conditions there is a way of relating the algebra obtained by the induction from $H$ to $G$ of a certain $H$-algebra with the algebra obtained by the idempotent $Tr^G_H(i)$ where $i$ is an idempotent fixed by $H$.  Given $A$ an interior $G$-algebra, we will denote the set of elements of $A$ fixed by $H$ as $A^H$ and $1_A$ as the multiplicative identity of $A$.

\vspace{5mm}

\begin{proposition}
\label{proposition:Thevenaz2} (Proposition 16.6, Th{\'e}venaz  \cite{thevenaz1995g}) Let $A$ be an interior $G$-algebra and let $H$ be a subgroup of $G$. Assume that there exists an idempotent $i \in A^H$ such that
$1_A = Tr^G_H(i)$ and $i^gi = 0$ for all $g \in G-H$. Then there is an isomorphism of interior G-algebras 
$$F : \mathrm{Ind^G_H}(iAi)\simeq A;$$ 
Given by  $x \otimes \hspace{1 mm} b \otimes y \mapsto x\cdot b \cdot y \hspace{2 mm} (x,y \in G, b \in iAi) $. 
\end{proposition}
 
\vspace{5mm}
  We use \Cref{proposition:Thevenaz2} to prove the following:
  
\begin{proposition}
\label{proposition:Block}
Given a block $b$ of $KG_2$, there is a block $\hat{\Phi}(b)$ of $KG'_2$ such that the following $K$-algebras $bKG_2$ and $\hat{\Phi}(b)KG'_2$ are isomorphic. 
\end{proposition}

\begin{proof}
Fix $\chi \in \Irr(N)$ up to conjugation by $G$, and let $e=e_\chi$ a primitive idempotent of $KN$, now fix a block $b \in Bl(G_2|e)$ such that by  Clifford \Cref{theorem:Clifford} $b=Tr^{G_2}_{H_2}(d)$ for some primitive idempotent $d$ of $KH_2$ where $H_2$ is the stabilizer of $e$ in $G_2$. Notice by \Cref{theorem:Stabilizer}, we have that $ \Phi:eKH_2\cong eKH'_2$ as $K$-algebras. Now the map $\Phi$ gives a bijection between $Bl(H_2|e)$ and $Bl(H'_2|e)$ such that $ \Phi:dKH_2\cong \Phi(d)KH'_2$ for each $d \in Bl(H_2|e)$. Moreover, Clifford \Cref{theorem:Clifford} tells us there is a bijection between $Bl(G|e)$ and $Bl(H|e)$. Thus, we obtain a bijection between $Bl(G_2|e)$ and $Bl(G'_2|e)$ by the map  $b \mapsto \hat{\Phi}(b)=Tr^{G'_2}_{H'_2}(\Phi(d)) $.

Let $A=bKG_2$ so $1_A=b=Tr^{G_2}_{H_2}(d)$. Note $d\cdot b=d$ so $d \in A^{H_2}$ and $d^gd = 0$ for $g \in G_2-H_2$ \cite[In proof of Lemma 6.8.4]{linckelmann2018block}. Thus, we have that $$dAd=dbKG_2d=dKG_2d=dKH_2d=dKH_2.$$ Thus by \Cref{proposition:Thevenaz2}, there is an isomorphism of interior $G_2$-algebra: 

$$\mathrm{Ind^{G_2}_{H_2}}(dKH_2)\cong bKG_2$$ 

Note by \Cref{lemma:Thevenaz} that $\mathrm{Ind^{G_2}_{H_2}}(dKH_2)\cong \mathrm{M}_{n}(dKH_2)$ as $K$-algebras, with $n=|G_2:H_2|=|G'_2:H'_2|$. Since, $ \Phi:dKH_2\cong \Phi(d)KH'_2$ conclude that

$$ \hat{\Phi}: bKG_2 \cong \mathrm{Ind^{G_2}_{H_2}}(dKH_2) \cong \mathrm{Ind^{G'_2}_{H'_2}}(\Phi(d)KH'_2) \cong \hat{\Phi}(b)KG'_2.$$
Thus we have that $bKG_2$ and $\hat{\Phi}(b)KG'_2$ are isomorphic as $K$-algebras.
\end{proof}

With the above $K$-algebra isomorphism, we define a map $\Psi: KG_2 \to KG'_2$ such that $\Psi(a)=\sum\limits_{b \in Bl(G_2)}\hat{\Phi}(a\cdot b)$ for $a \in KG_2$ and show that this map is an isomoprhism of $K$-algebras.

\vspace{5mm}

\begin{theorem}
\label{theorem:Final}
Let $G_2=\mathbb{G}(\mathcal{O}_2)$ and $G'_2=\mathbb{G}(\mathcal{O}'_2)$, be the group of points of any reductive group scheme $\mathbb{G}$ over $\mathbb{Z}$ such that $p$ is very good for $\mathbb{G} \times \mathbb{F}_q$. There exists an isomorphism of group algebra $K[\mathbb{G}(\mathcal{O}_2)] \cong  K[\mathbb{G}(\mathcal{O}'_2)]$, where $K$ is a sufficiently large field of  characteristic different from $p$. 
\end{theorem}

\begin{proof}
Define the map $\Psi: KG_2 \to KG'_2$ by  $\Psi(x)=\sum\limits_{b \in Bl(G_2)}\hat{\Phi}(xb)$ for $x \in KG_2$. First, we show that $\Psi$ is an algebra homomorphism. Note that by definition it is a linear map since each of the $\hat{\Phi}$ from \Cref{proposition:Block} are. Now given $x$ and $y$ both in $KG_2$, we want to show that $\Psi(xy)=\Psi(x)\Psi(y)$. To prove this it is sufficient to show that different blocks of $G_2$ are mapped to different blocks of $G'_2$. Because if they are, we have that  $$\Psi(x)\Psi(y)=\sum\limits_{b \in Bl(G_2)}\hat{\Phi}(xb)\cdot \sum\limits_{b \in Bl(G_2)}\hat{\Phi}(yb)=\sum\limits_{b \in Bl(G_2)}\hat{\Phi}(xb)\cdot  \hat{\Phi}(yb).$$

Since the products of elements of different blocks is zero. Moreover by $\hat{\Phi}$ being a $K$- algebra homomorphism for each block $b$ we have: 

$$\Psi(x)\Psi(y)=\sum\limits_{b \in Bl(G_2)}\hat{\Phi}(xb)\cdot  \hat{\Phi}(yb)= \sum\limits_{b \in Bl(G_2)}\hat{\Phi}(xbyb)=\sum\limits_{b \in Bl(G_2)}\hat{\Phi}(xyb)=\Psi(xy)$$

Now to show that different blocks of $G_2$ are mapped to different blocks of $G'_2$ we look at two cases. Let $b$ and $c$ be two different blocks of $G_2$:

\vspace{2mm}

Case 1: Assume $b$ and $c$ cover the same block $e$ of $KN$.
Thus, by Clifford \Cref{theorem:Clifford} $b=Tr^{G_2}_{H_2}(d_1)$ and $c=Tr^{G_2}_{H_2}(d_2)$ such that $d_1 \cdot d_2=0$. Recall by \Cref{theorem:Stabilizer} that the map $\Phi: eKH_2 \cong eKH'_2$ is a $K$-algebra isomorphism, thus $\Phi(d_1) \Phi(d_2)=0$. Applying Clifford \Cref{theorem:Clifford} for the block of $KG'_2$ over $e$, we have that  $Tr^{G'_2}_{H'_2}(\Phi(d_1))\neq Tr^{G'_2}_{H'_2}(\Phi(d_2))$, since $ \Phi(d_1)\neq \Phi(d_2)$. Thus $\hat{\Phi}(b)\neq \hat{\Phi}(c)$.

\vspace{2mm}

Case 2: Assume $b$ and $c$ cover different blocks $e_1$ and  $e_2$ of $KN$ respectively. Note to prove that $\hat{\Phi}(b)\neq \hat{\Phi}(c)$, it is enough to show that they cover different blocks of $KN$ as blocks of $G'_2$. Assume the opposite. By definition of the map $\hat{\Phi}$, we also have that $\hat{\Phi}(b)$ and $\hat{\Phi}(c)$ cover the blocks $e_1$ and  $e_2$ of $KN$ respectively. Thus by \cite[Proposition 6.8.2]{linckelmann2018block}, we have that  $e_1$ and  $e_2$ are $G'_2$-conjugated blocks. Since $KN$ is a semisimple algebra i.e all the blocks have defect zero, $e_1$ and  $e_2$ are defined by a unique ordinary character $\chi$ and $\tau$ of $KN$. Thus $e_1$ and  $e_2$ are $G'_2$-conjugated blocks if and only if $\chi$ and $\tau$ are $G'_2$-conjugated characters. By \Cref{lemma:Characters}, we can conclude that $\chi$ and $\tau$ are $G_2$-conjugated and thus so are $e_1$ and  $e_2$. This is a contradiction with the fact that $b$ and $c$ cover different blocks $e_1$ and  $e_2$ of $KN$.

Thus, we can conclude that different blocks of $G_2$ are mapped to different blocks of $G'_2$. Therefore, we have an algebra homomorphism from $\Psi: KG_2 \to KG'_2$ and since each $\hat{\Phi}$ is injective and maps different blocks to different blocks, we can conclude $\Psi$ is an injective map. Thus by dimension reasons, $\Psi$ is also surjective. Therefore, $KG_2 \cong KG'_2$ as $K$-algebras.
\end{proof}

\vspace{2mm}

\bibliographystyle{plain}
\bibliography{THM_1.bib}

\end{document}